\documentclass[11pt, reqno]{amsart}

\usepackage[utf8]{inputenc}
 
\usepackage{amsfonts, amssymb, amscd, amsthm, amsmath}
\usepackage{hyperref}
\usepackage{enumitem}

\usepackage{verbatim}
\usepackage{amssymb}
\usepackage{mathrsfs}
\usepackage{graphicx}
\usepackage{cite}
\usepackage[all]{xy}
\usepackage{tikz-cd}
\usepackage{subfiles}
\usepackage[toc,page]{appendix}
\usepackage{mathtools}

\usepackage{hyperref}
\hypersetup{colorlinks, linkcolor={violet}, citecolor={violet}}

\newcommand{\Zz}{\mathbb{Z}}
\newcommand{\Cc}{\mathbb{C}}
\newcommand{\Pp}{\mathbb{P}}
\newcommand{\Rr}{\mathbb{R}}
\newcommand{\Qq}{\mathbb{Q}}
\newcommand{\Nn}{\mathbb{N}}

\newcommand{\spec}{\operatorname{Spec}}

\newcommand{\pet}{\operatorname{pet}}
\newcommand{\PET}{\operatorname{PET}}

\newcommand{\Supp}{\operatorname{Supp}}

\newcommand{\Oo}{\mathcal{O}}

\newtheorem{theorem}{Theorem}[section]
\newtheorem{lemma}[theorem]{Lemma}
\newtheorem{proposition}[theorem]{Proposition}
\newtheorem{definition}[theorem]{Definition}
\newtheorem{example}[theorem]{Example}

\newtheorem{remark}[theorem]{Remark}
\newtheorem{conjecture}[theorem]{Conjecture}

\numberwithin{equation}{section}

\begin{document}

\title[Fujita's log spectrum conjecture]{Fujita's conjecture on iterated accumulation points of pseudo-effective thresholds}

\begin{abstract}
We show that $k$-th iterated accumulation points of pseudo-effective thresholds of $n$-dimensional varieties are bounded by $n-k+1$.
\end{abstract}

\author{Zhan Li}
\address{Department of Mathematics, Southern University of Science and Technology, 1088 Xueyuan Rd, Shenzhen 518055, China} \email{lizhan@sustech.edu.cn}

\date{\today}

\maketitle

\tableofcontents

\section{Introduction}\label{sec: introduction}

We work over complex numbers. The purpose of this paper is to establish a version of Fujita's log spectrum conjecture on iterated accumulation points of pseudo-effective thresholds. For all triples $(X, \Delta, M)$, where $(X, \Delta)$ is an $n$-dimensional log smooth variety with $\Delta$ a reduced divisor and $M$ an ample Cartier divisor, let $\PET_n$ be the set of the pseudo-effective threshold of $M$ with respect to $(X, \Delta)$ (see Definition \ref{def: pet}). For a subset $S \subseteq \Rr$, let $\lim^1 S$ denote the set of accumulation points of $S$, and let the set of the $k$-th iterated accumulation points of $S$ be $\lim^k S \coloneqq \lim^1(\lim^{k-1} S)$ for any $k \in \Nn$.

\begin{conjecture}[{Fujita's log spectrum conjecture \cite[(3.7)]{Fuj96}}]\label{conj: Fujita's log spectrum conjecture} We have $\lim^{k}(\PET_n) \leq n-k$ for any positive integer $k\leq n$ .
\end{conjecture}

Here $\lim^{k}(\PET_n) \leq n-k$ means that each element in $\lim^{k}(\PET_n)$ is less or equal to $n-k$. In this paper, we obtain an upper bound which is $1$ bigger than the conjectured one.

\begin{theorem}\label{thm: Fujita}
We have $\lim^{k}(\PET_n) \leq n-k+1$ for any positive integer $k\leq n$.
\end{theorem}

In fact, the above result is a direct consequence of the following more general statement which allows the coefficients of the boundary divisor $\Delta$ to vary in a DCC set (not just $1$). 

\begin{theorem}\label{thm: main}
Let $I \subseteq [0,1]$ be a DCC set such that $1$ is the only possible accumulation point of $I$, then $\lim^k (\PET_n(I)) \leq n-k+1$ for any positive integer $k\leq n$.
\end{theorem}

For the meaning of the DCC set see Section \ref{sec: preliminaries}, and for the definition of $\PET_n(I)$ see \eqref{eq: pet(I)}. Under the above conditions, the upper bound is sharp and the assumption that $1$ is the only possible accumulation point of $I$ cannot be removed. In fact, for $i \in \Nn$, consider $X_i=\Pp^1$ and $K_{X_i}+B_i+c_iM_i \equiv 0$. Suppose that $M_i $ is a closed point and $B_i$ is a closed point with coefficient $\frac{i-1}{i}$, then $c_i = \frac{i+1}{i}$ whose accumulation point is $1$. On the other hand, if we choose the coefficient of $B_i$ to be $\frac{i-1}{2i}$ (hence $\frac 1 2$ is an accumulation point of $I$), then $c_i =\frac{3i+1}{2i}$ and the accumulation point of $c_i$ is $\frac 3 2>1$. This example can be generalized to any $n$ and $0<k \leq n$ (see Example \ref{eg: the bound is sharp}). Previously, such result is only known for $k=n-1, n$ (\cite[Proposition 1.3]{HL18b}) which more or less corresponds to the surface case.

\medskip

Finally, we say a few words about the history of Fujita's log spectrum conjecture, in particular those are pertinent to the current work. Fujita proposed two conjectures on the behavior of pseudo-effective thresholds (see \cite{Fuj92} and \cite[(3.2) (3.7)]{Fuj96}) which are analogies to the Shokurov's conjecture on log canonical thresholds. One conjecture predicts that the set of pseudo-effective thresholds is ACC, the other conjecture is Conjecture \ref{conj: Fujita's log spectrum conjecture}. The first conjecture has been settled affirmatively in \cite{DiC16, DiC17}. Besides, there is a series of related works from adjunction theory (see \cite{Fuj90, BS95}). 

\medskip

In \cite{HL17}, we view the first conjecture from the perspective of generalized polarized pairs and establish a result for a wider class of varieties. In some sense, such result is optimal. The advantage of this new perspective is that it automatically takes care of the singularities introduced by the testing divisors. Towards Conjecture \ref{conj: Fujita's log spectrum conjecture}, the same perspective is adopted in \cite{HL18b} combining with the observation that there is a striking similarity between pseudo-effective thresholds and log canonical thresholds. Hence, there is no surprise that the induction result in \cite{HL18b} is established parallelly to that in \cite{HMX14}. However, in order to get the full statement as Theorem \ref{thm: main}, there is a difficulty in the fibration case. \cite{HL18b} was submitted on May 2018 without going to the arXiv. The new idea of the current work is to use Nakayama's subadditivity on numerical dimensions to tackle the fibration case (see Theorem \ref{thm: Nakayama}). Combining with a modified induction result of \cite{HL18b}, we are eventually able to establish Theorem \ref{thm: Fujita} and Theorem \ref{thm: main}. 

\medskip

The paper is organized as follows. In Section \ref{sec: preliminaries}, we give background materials. Theorem \ref{thm: Fujita} and Theorem \ref{thm: main} are proven in Section 
\ref{sec: proof of thm}.
\medskip

\noindent\textbf{Acknowledgements}.
This paper is a continuation of the previous joint work with Jingjun Han \cite{HL18b}. The author thanks Chen Jiang for simplifying the original argument of Proposition \ref{prop: bound on Fano index}. This work is partially supported by NSFC Grant No.11601015 and a starting grant from SUSTech.

\section{Preliminaries}\label{sec: preliminaries}
Throughout the paper, $\Zz$ and $\Nn$ denote the set of integers and the set of positive integers respectively. 

\subsection{Pseudo-effective thresholds, ACC/DCC sets.}

\begin{definition}[Pseudo-effective threshold]\label{def: pet}
If $M$ is an $\Rr$-Cartier divisor, we define the \emph{pseudo-effective threshold} of $M$ with respect to $(X, \Delta)$ to be
\[
\pet(X, \Delta; M) \coloneqq \inf\{t \in \Rr_{\geq 0} \mid K_X+\Delta+tM \text{~is effective}\}.
\] 
\end{definition}

By convention, $\pet(X, \Delta; M) = +\infty$ if $K_X+\Delta+t M$ is not pseudo-effective for any $t \in \Rr_{\geq 0}$.

\medskip

A set of real numbers is ACC (resp. DCC) if it satisfies the ascending chain condition (resp. descending chain condition). Let $I \subseteq [0,1]$, we define
\begin{equation}\label{eq: pet(I)}
\begin{split}
\PET_n(I) \coloneqq \{ & \pet(X, \Delta; M) \mid (X, \Delta) \text{~is lc}, \text{~coefficients of }\Delta \text{~are in~} I, \\
&M \text{~is an ample Cartier divisor}, \dim X=n\}.
\end{split}
\end{equation}

\begin{remark}
In the definition of $\PET_n(I)$ in \cite{HL18b}, we only require that $M$ is a \emph{nef and big} Cartier divisor. Here we need \emph{ampleness} of $M$ to exploit the fibration defined by $M$. Nevertheless, it is still curious to know if Theorem \ref{thm: main} still holds in the nef and big setting.
\end{remark}

\subsection{Generalized polarized pairs.}

\begin{definition}[{Generalized polarized pair, \cite[Definition 1.4]{BZ16}}]
A generalized polarized pair $(X'/Z, B'+M')$ consists of a normal variety $X'$ equipped with projective morphisms 
\[
X \xrightarrow{f}X' \to Z,
\] where $f$ is birational and $X$ is normal, an $\Rr$-boundary $B' \geq 0$, and an $\Rr$-Cartier divisor $M$ on $X$ which is nef$/Z$ such that $K_{X'}+B'+M'$ is $\Rr$-Cartier, where $M'\coloneqq f_*M$. We call $B'$ the boundary part and $M$ the nef part.
\end{definition}

From the definition, we see that $X$ could be replaced with any log resolution over $X$, and $M$ could be replaced with the pullback of $M$ accordingly. We can define the generalized log discrepancy of a divisor $E$ over $X'$ by considering a high enough model $X$ which contains $E$ (say a resolution as above). Let
\[
K_X+B+M=f^*(K_{X'}+B'+M'),
\] then the \emph{generalized log discrepancy} of $E$ is defined as (\cite[Definition 4.1]{BZ16})
\begin{equation}\label{eq: g-log discrepancy}
a(E,X',B' + M') = 1-{\rm mult}_EB.
\end{equation} We say that $(X',B'+M')$ is generalized log canonical (resp. generalized kawamata log terminal) if the generalized log discrepancy of any prime divisor is $\geq 0$ (resp. $>0$). For simplicity, we write g-lc (resp. g-klt) for generalized log canonical (resp. generalized kawamata log terminal).

\subsection{The sets $\mathfrak{N}_n(I,c), \mathfrak{K}_n(I,c)$ and $N_n(I), K_n(I)$.}

Let $I \subseteq [0,1]$, we define
\begin{equation}\label{eq: I_+}
I_+ \coloneqq \{\sum n_i a_i \leq 1 \mid a_i \in I, n_i \in \Nn\}.
\end{equation} For a divisor $\Delta$, we write $\Delta \in I$ if the coefficients of $\Delta$ lie in $I$. Because the coefficient set $I$ may change after adjunctions, for $c\in \Rr_{\geq 0}$, we define
\begin{equation}\label{eq: coefficients of adjunction}
\begin{split}
&D(I) \coloneqq \{\frac{m-1+f}{m} \leq 1 \mid m \in \Nn, f \in I_+\}, and\\
&D_c(I) \coloneqq \{\frac{m-1+f+kc}{m} \leq 1 \mid m, k \in \Nn, f \in I_+\}.
\end{split}
\end{equation}

Suppose that $(X', B'+M')$ is a generalized polarized pair with data $X \xrightarrow{f} X' \to \spec \Cc$ and $M$ the nef part ($M$ may not be effective), where $f_*M=M'$. Let $c\in \Rr_{\geq 0}$, then $(X', B'+cM')$ is said to satisfy condition ($\dagger$) if the following properties hold:

\medskip

\begin{quote}
~($\dagger$)  
\medskip
\begin{enumerate}
\item $(X', B' + cM')$  is g-lc,
\item $K_{X'}+B' + cM' \equiv 0$, and  $B' \in D(I) \cup D_c(I)$, 
\item $M' = f_*M$ with $M$ semi-ample and Cartier, and $M'$ $\Qq$-Cartier,
\item suppose that $\phi_M: X \to Z$ is the fibration defined by a sufficiently divisible multiple of $M$, then $M=\phi_M^*M_Z$ is a pullback of an ample \emph{Cartier} divisor $M_Z$, and
\item if $M \equiv 0$, then at least one coefficient of $B'$ lies in $D_c(I)$.
\end{enumerate} 
\end{quote}

\begin{remark}
In \cite{HL18b}, we only require $M$ to be \emph{nef} in (3), and (4) is a newly added assumption in order to guarantee the induction argument.
\end{remark}

Recall the following sets defined in \cite{HL18b} under the above modified definition of condition $(\dagger)$. 

\begin{equation}\label{eq: set N}
\begin{split}
\mathfrak{N}_n(I,c) = &\{(X', B' + cM') \mid \dim X =n, \\
&\text{~and~} (X', B' + cM') \text{~satisfies condition ($\dagger$)} \}.
\end{split}
\end{equation}

Notice that when $M \equiv 0$ (or equivalently $M'\equiv 0$), this generalized polarized pair is just the lc pair $\mathfrak{N}_n(I,c)$ defined in \cite{HMX14} Page 559. In the same fashion, we define
\[
\begin{split}
\mathfrak{K}_n(I,c) = &\{(X', B' + cM') \mid (X', B' + cM') \in \mathfrak{N}_n(I,c)  \\
&\text{~is g-klt, and~} \rho(X')=1\}.
\end{split}
\] The corresponding subsets of real numbers are
\begin{equation}\label{eq: sets}
\begin{split}
N_n(I) =  &\{c \in \Rr \mid \mathfrak{N}_m(I,c) \neq \emptyset, m \leq n \},\\
K_n(I) =  &\{c \in \Rr \mid \mathfrak{K}_m(I,c) \neq \emptyset, m \leq n \}.
\end{split}
\end{equation} Notice that in $N_n(I)$ and $K_n(I)$, we also consider varieties of dimensions less than $n$.

\subsection{Subadditivity of numerical dimensions.}

Let $D$ be an $\Rr$-divisor, recall that $\kappa_\sigma(D)$ is the Nakayama's numerical dimension of $D$ (\cite[\S V Definition 2.5]{Nak04}) which is defined by
\[
\kappa_\sigma(D) \coloneqq \max\{\sigma(D; A) \mid A \text{~is a divisor}\},
\] where $\sigma(D; A)=-\infty$ if $H^0(X, A+\lfloor mD\rfloor) \neq 0$ only for finitely many $m \in \Nn$, otherwise
\[
\sigma(D; A) \coloneqq \max\{k \in \Zz_{\geq 0} \mid \overline{\lim}_{m \to \infty} m^{-k}h^0(X, A+ \lfloor mD\rfloor) > 0) \}.
\]  Besides, one defines
\[
\kappa_\sigma(D; X/Y) \coloneqq \kappa_\sigma(D|_{X_z})
\] on the fiber $X_z$, where $z \in Z$ is a general point. Here we assume that $X \to Z$ has connected fibers. In general, one considers a general fiber of its Stein factorization (see \cite[\S V Notation 2.24]{Nak04}). The following lemma is well-known.

\begin{lemma}\label{le: pseudo-effective}
Suppose that $X$ is a smooth projective variety and $D$ is an $\Rr$-Cartier divisor. Then $\kappa_\sigma(D) = -\infty$ iff $D$ is not pseudo-effective.
\end{lemma}
\begin{proof}
In \cite[\S V Remark 2.6]{Nak04}, it is shown that $\kappa^-_\sigma(D) = -\infty$ iff $D$ is not pseudo-effective (see \cite[\S V Definition 2.5]{Nak04} for $\kappa^-_\sigma(D)$). Moreover, $\kappa_\sigma(D) = -\infty$ iff for any $A$, $H^0(X, A+\lfloor mD\rfloor) \neq 0$ only for finitely many $m \in \Nn$. By definition, this is equivalent to $\kappa^-_\sigma(D) = -\infty$.
\end{proof}

\begin{theorem}[{\cite[\S V Theorem 4.1(1)]{Nak04}, \cite[(3.3) and (3.4)]{Fuj19}}]\label{thm: Nakayama}
Let $g: Y \to Z$ be a fiber space (with connected fibers) from a normal projective variety into a non-singular projective variety, $\Delta$ an effective $\Rr$-divisor of $Y$ such that $K_Y +\Delta$ is $\Rr$-Cartier and $(Y, \Delta)$ is lc over a non-empty open subset of $Z$. Let $D$ be an $\Rr$-Cartier divisor of $Y$ such that $D-(K_{Y/Z} + \Delta)$ is nef. Then for any $\Rr$-divisor $Q$ of $Z$, 
\begin{equation}\label{eq: kodaia 1}
\kappa_\sigma(D+f^*Q) \geq \kappa_\sigma(D; Y/Z)+\kappa(Q) {\rm~ and}
\end{equation} 
\begin{equation}\label{eq: kodaia 2}
\kappa_\sigma(D+f^*Q) \geq \kappa(D; Y/Z)+\kappa_\sigma(Q),
\end{equation} where $\kappa(-)$ denotes the Kodaira dimension of the divisor.
\end{theorem}

\begin{remark}
See \cite[Remark 3.8]{Fuj17} and \cite{Fuj19} for further discussions on this result.
\end{remark}

\section{Proofs of Theorems}\label{sec: proof of thm}

First, we need the following induction characterization of $k$-th iterated accumulation points of $\PET_n(I)$. This is a modified version of the main result of \cite{HL18b} under current settings. 

\begin{theorem}[{\cite[Theorem 1.2]{HL18b}}]\label{thm: accumulation point of pet}
Let $I \subseteq [0,1]$ be a DCC set such that $I=I_+$. Assume that $1\in I$ with $1$ the only possible accumulation point, then for any $1 \leq k \leq n-1$, $\lim^k (\PET_n(I)) \subseteq K_{n-k}(I)$.
\end{theorem}
\begin{proof}[Sketch of Proof]
The differences between \cite[Theorem 1.2]{HL18b} and the above result lie in the different meanings of $\PET_n(I)$ and $K_{n-k}(I)$. Under our assumptions, $M$ in the definition of $\PET_n(I)$ (see \eqref{eq: pet(I)}) is an ample Cartier divisor instead of just nef and big, and $M$ in the definition of $K_{n-k}(I)$ (see \eqref{eq: sets}) satisfies additional assumptions (3) and (4) in the condition $(\dagger)$. That is, $M$ is a semi-ample Cartier divisor and for the morphism $\phi_M$ defined by $M$, $M=\phi_M^*M_Z$ for an ample Cartier divisor $M_Z$. 

\medskip

However, the argument for \cite[Theorem 1.2]{HL18b} works without any essential changes. The reason is that $M$ in $K_{n-k}(I)$ is obtained by combinations of the following three actions:

\begin{enumerate}[label=(\alph*)]
\item pullback an ample Cartier divisor $A$, 
\item restriction on a general fiber $F$ of some fibration, and
\item restriction on a divisor $S$. 
\end{enumerate}

Suppose that $M' =f^*A$ is a pullback of an ample Cartier divisor $A$ on $Z'$ through $f: X \to Z'$. In case (b), if $F$ is the fiber, then $Z$ of ($\dagger$)(4) is the target of the morphism $\phi_{M'|_F}$ defined by $M'|_F$, which is just the normalization of $f(F)$. Hence $M \coloneqq M'|_F = (\phi_{M'|_F})^*(A|_{Z})$ is semi-ample and Cartier, and $M_Z=A|_{Z}$ is  ample and Cartier. In case (c), $Z$ of ($\dagger$)(4) is the target of the morphism $\phi_{M'|_S}$ defined by $M'|_S$, which is just the normalization of $f(S)$. Hence $M \coloneqq M'|_S=(\phi_{M'|_S})^*(A|_Z)$ is semi-ample and Cartier, and $M_Z=A|_Z$ is ample and Cartier. Hence the additional requirements on $M$ and $M_Z$ preserve under the above three actions.

\medskip

For the readers' convenience, we sketch the proof of \cite[Theorem 1.2]{HL18b}, leaving technical details to the original argument.

\medskip

First, we claim that $N_n(K)=K_n(I)$ (see \cite[Lemma 3.3]{HL18b}).

\begin{proof}[Proof of the Claim] Only ``$\subseteq$'' is non-trivial. For $(X', B'+cM') \in \mathfrak{N}_n(I,c)$, if $M' \equiv 0$, this is just \cite[Lemma 11.4]{HMX14}, hence we can assume that $M' \not\equiv 0$. After taking a generalized dlt modification, we can assume that it is $\Qq$-factorial and $(X', 0)$ is klt. Run a $(K_{X'}+B')$-MMP, then it terminates to a Mori fiber space $f: X'' \to Z''$. If $\lfloor B' \rfloor =0$, we restrict to a general fiber of $f$ and get a g-klt pair. If $\dim Z''>0$, we obtain the result by induction on dimensions. Otherwise, we get a Picard number $1$ variety and $c\in K_n(I)$ by definition. We emphasize that in the above process, we only use actions (a) (taking a higher log resolution $W$ dominating $X'$ and $X''$, and pullback to $W$) and (b), and hence the induction hypothesis is preserved. If $\lfloor B' \rfloor \neq 0$, we do the same thing when $\dim Z''>0$. When $\dim Z''=0$, if $\lfloor B' \rfloor$ is not contracted in the above MMP, we do adjunction on an irreducible component of $\lfloor B' \rfloor$ and obtain the claim by induction on dimensions. This step uses action (c). If a component $S'$ of $\lfloor B' \rfloor$ is contracted in the above MMP, one can show that $M'|_{S'} \not\equiv 0$ at this step, and the adjunction on $S'$ gives the desired claim by induction again. This step also uses action (c).
\end{proof}

Next, we claim that $\lim^1 \PET_n(I) \subseteq \lim^1N_n(I)$ (see \cite[Proposition 3.4]{HL18b}).

\begin{proof}[Proof of the Claim] 
Suppose that there exists a sequence of lc pairs $({X_i}, \Delta_i)$ and ample Cartier divisors $M_i$ such that $c_i =\pet(X_i, \Delta_i; M_i)$ with $\lim c_i =c$ the accumulation point. After taking a dlt modification, we choose $0<\epsilon_i \ll 1$ and $A_i$ an ample divisor, such that $(X_i, \tilde \Delta_i)$ is klt where $\tilde \Delta_i\sim_\Rr \Delta_i + \epsilon_iA_i$. The new $c_i' =\pet(X_i, \tilde \Delta_i; M_i)$ has the same accumulation point $c$. We run an MMP for the g-klt pair $(X_i, \tilde \Delta_i+c_i'M_i)$. The resulting model gives a fibration $Y_i \to Z_i$. If $\dim Z_i>0$, we take a general fiber $F_i$. Consider $\tau_i = \pet(F_i, \Delta_{Y_i}|_{F_i}; M_{Y_i}|_{F_i})$ and the generalized log canonical threshold $\tau_i'$ of  $M_{Y_i}|_{F_i}$ with respect to $(F_i, \Delta_{Y_i}|_{F_i})$. If $\tau_i' \geq \tau_i$, we are done by induction on dimensions because $\lim \tau_i =c$. Otherwise, there exists a generalized lc place $S_i$. We can first extract $S_i$ and then do adjunction on $S_i$ and finally restrict to the general fiber of $S_i$ mapping to the normalization of its generalized lc center. An induction on dimensions establishes the claim. Notice that even though the pseudo-effective thresholds may change in each step, its limit is the same. Moreover, we emphasize that the above process only uses the actions (a) (b) and (c). If $\dim Z_i =0$, then we run a $(K_{Y_i}+\Delta_{Y_i})$-MMP, and it terminates to a Mori fiber space $W_i \to Z'_i$. A similar argument as above shows the claim.
\end{proof}

Finally, we claim that if $\{c_i \mid c_i \in N_n(I)\}$ has an accumulation point $c$, then $c \in N_{n-1}(I)$ (see \cite[Proposition 3.6]{HL18b}).

\begin{proof}[Proof of the Claim] 
By the first claim, we can assume that there exists Fano varieties $X'_i$ such that $K_{X'_i}+\Delta_i'+c_i M_i' \equiv 0$ satisfying condition $(\dagger)$. After several reduction steps, we come to the main observation that $\{X_i'\}$ cannot be $\epsilon$-lc for any $\epsilon>0$. Otherwise, the BAB conjecture proved by Birkar \cite{Bir16b} would imply the boundedness of $\{X_i'\}$ and thus $\{c_i\}$ cannot have an accumulation point. Then we can extract a divisor $A'_i$ whose generalized log discrepancy is at most $\epsilon$, and assume $K_{X'_i}+A'_i+\Delta'_i+c_i M'_i \equiv 0$. Run a $(K_{X'_i}+\Delta'_i+c_i M'_i)$-MMP, we get a Mori fiber space and by the induction on dimensions, we can assume that the base is a point. Hence we can still assume that $\rho(X'_i)=1$. If $\lfloor A'_i \rfloor \neq 0$, we do adjunction on an irreducible component of $\lfloor A'_i \rfloor$ and finish the proof by induction on dimensions again. Otherwise, set $T_i' = \Supp A'_i$. By ACC for generalized lc thresholds (\cite[Theorem 1.5]{BZ16}), we can assume that $(X'_i, T'_i+\Delta'_i+cM'_i)$ is g-lc. Let $c_i'$ be the generalized log canonical threshold of $M'_i$ with respect to $(X'_i, T'_i+\Delta'_i)$. If $c_i'<c_i$, then there exists a generalized lc place $S_i'$ which intersects the preimage of $M'_i$ non-trivially. Then do adjunction on $S_i'$ and restrict to a general fiber over the normalization of its generalized lc center, we finish the proof by induction on dimensions. Hence, we can assume that $(X'_i, T'_i+\Delta'_i+c_iM'_i)$ is g-lc. If $K_{X_i'}+T'_i+\Delta'_i+cM'_i \equiv 0$, we do adjunction on an irreducible component of $T_i'$ and obtain the result by induction on dimensions. Otherwise $K_{X_i'}+T'_i+\Delta'_i+cM'_i$ can only be anti-ample for infinite $i$ (the ample case would contradict the global ACC for generalized lc pairs  (\cite[Theorem 1.6]{BZ16})). Suppose that $c_i''$ satisfies $K_{X_i'}+T'_i+\Delta'_i+c_i''M'_i \equiv 0$, then $c<c_i'' \leq c_i$ and $\lim c''_i = c$. By adjunction on an irreducible component of $T_i'$, we complete the proof. Notice that the above process only uses actions (a) (b) and (c).
\end{proof}

Put the above together and use  the third claim repetitively, we have
\[
\lim{^k} \PET_n(I) \subseteq \lim{^k} N_n(I) \subseteq N_{n-k}(I) = K_{n-k}(I).
\]
\end{proof}

\begin{remark}
One may wonder where do the conditions $(\dagger) (2)$ and $(5)$ come from. In fact, these are the results of an adjunction on an irreducible component $S'$ of $\lfloor B' \rfloor$ for the g-lc pair $(X', B'+cM')$. Suppose that $f: X \to X'$ is a sufficiently high log resolution, viewing it as a generalized polarized pair with data $f$ and $M$, and $S$ is the strict transform of $S'$. Assume that $B\in I$, if $M'|_{S'} \not\equiv 0$, but $M|_S \equiv 0$, then in the (generalized) adjunction $(K_{X'}+B'+cM')|_{S'} = K_{S'}+B_{S'}+cM_{S'}$, at least one coefficient of $B_{S'}$ lies in $D_c(I)$. Indeed, we have $f^*M' = M+E$ with $E \geq 0$ an $f$-exceptional divisor. Then $f_{S}^*(M'|_{S'})=(f^*M')|_S = M|_S+E|_S$ and $M|_S\equiv 0$ imply that $0 \not\equiv M'|_{S'} \equiv {f_{S}}_*(E|_S)$. Because $B_{S'} \coloneqq {f_{S}}_*((B-S)|_S+cE|_S)$, we have the claim. For details, see \cite[Lemma 3.2]{HL18b}.
\end{remark}

The proof of the following proposition essentially uses the same idea of \cite[Corollary 1.6]{Sho85}. Such result should be well-known in the literature in variant forms (cf. \cite{Hor10}). 

\begin{proposition}\label{prop: bound on Fano index}
Let $(X', B'+M')$ be a generalized polarized pair such that $B'$ is the boundary part. Suppose that $M'$ is the push-forward of a nef and big Cartier divisor $M$ and let $\tau = \pet(K_{X'}+B', M')$, then $\tau \leq \dim X+1$.
\end{proposition}
\begin{proof}
By taking a sufficiently high log resolution of $(X', B'+M')$ and replacing $M$ with its pullback. We can assume that there exists a birational morphism $f: X \to X'$ from a smooth variety $X$, and $M$ on $X$ is a nef and big Cartier divisor. Let $\dim X=n$, then the Euler character $\chi(kM)=\sum_{i=0}^{n} (-1)^i h^i(X, kM)$ is a polynomial in $k \in \Zz$ with leading term $\frac{M^n}{n!}k^n$. In particular, $\chi(kM) \neq 0$.

\medskip

Suppose that $\tau > n+1$. For $k \in \Zz_{<0}$ and $i<n$, we have
\begin{equation}\label{eq: 1}
H^i(X, kM)=H^{n-i}(X, K_X-kM)^\vee = 0
\end{equation} by Kawamata-Viehweg vanishing theorem. When $i=n$, $H^n(X, kM)= H^0(X, K_X-kM)^\vee$. If $K_X-kM$ is pseudo-effective, so are $f_*(K_X-kM) = K_{X'}-kM'$ and $K_{X'}+B'-kM'$ by $B' \geq 0$. Hence $K_X-kM$ is not pseudo-effective for $-k<\tau$. In particular, $H^n(X, kM)=0$ for $k\geq -(n+1)>-\tau$. This shows that $\chi(kM)=0$ for $n+1$ integers $k =-(n+1), \ldots, -1$. This is a contradiction to $\chi(kM) \neq 0$.
\end{proof}

\begin{remark}
The above result does not need $(X', B'+\tau M')$ to be g-lc.
\end{remark}

Recall that a proper morphism between normal varieties $g : Y \to Z$ is called a fibration if $f_*\Oo_Y=\Oo_Z$. 

\begin{proposition}\label{prop: num Iitaka}
Suppose that $Y, Z$ are normal projective varieties with $\dim Z>0$ and $g: Y \to Z$ is a fibration. Let $M_Z$ be a big $\Rr$-Cartier divisor on $Z$ and $M = g^*M_Z$. Suppose that $(Y, B)$ is a klt pair such that for a general point $z\in Z$, $(K_Y+B)|_{Y_z}$ is pseudo-effective. If $\tau \coloneqq \pet(Y, B; M)>0$, then $\pet(\tilde Z, 0; \tilde M_Z) \geq \tau$, where $\tilde Z \to Z$ is any resolution of $Z$ and $M_{\tilde Z}$ is the pullback of $M_Z$.  
\end{proposition}
\begin{proof}
Let $p: \tilde Z \to Z$ be any resolution and $q: \tilde Y \to Y$ be a log resolution of $(Y, B)$ such that $\tilde g: \tilde Y \to \tilde Z$ is a morphism satisfying $g \circ q= p \circ \tilde g$. Because $Z$ and $\tilde Z$ are isomorphism over a non-empty open set, $\tilde g$ has connected fibers over the generic point of $\tilde Z$. Let $\tilde Y \to \tilde Y_{\rm Stein} \to \tilde Z$ be the Stein factorization. Then $\tilde Y_{\rm Stein} \to \tilde Z$ is a finite morphism which is an isomorphism over the generic point. As $\tilde Z$ is normal, $\tilde Z = \tilde Y_{\rm Stein}$. Hence $\tilde Y \to \tilde Z$ is a fibration.

\medskip

As $\tilde M_Z$ is big, $\tau_0 \coloneqq \pet(\tilde Z, 0; \tilde M_Z) \geq 0$. Hence, if $\tau_0  < \tau$, there exists $\tau_0'>0$, such that $0 \leq \tau_0 < \tau_0' < \tau$. Then $K_Y+B+\tau'_0 M$ is not pseudo-effective. Set
\begin{equation}\label{eq: resolution}
K_{\tilde Y} + \tilde B_{\tilde Y} + \tau'_0 M_{\tilde Y} = q^*(K_Y+B+\tau'_0 M)+ E,
\end{equation} where $M_{\tilde Y} = q^*M$, and $E \geq 0$ is a $q$-exceptional divisor which does not contain any component of $\tilde B_{\tilde Y}$. Thus $(\tilde Y, \tilde B_{\tilde Y})$ is also klt. Besides, $K_{\tilde Y} + \tilde B_{\tilde Y} + \tau'_0 M_{\tilde Y}$ is not pseudo-effective, otherwise, 
\[
q_*(q^*(K_Y+B+\tau'_0 M)+ E) = K_Y+B+\tau'_0 M
\] is also pseudo-effective. Apply \eqref{eq: kodaia 1} in Theorem \ref{thm: Nakayama}  to $\tilde g: \tilde Y \to \tilde Z$ with $\Delta = \tilde B_{\tilde Y}$, $D=K_{\tilde Y/ \tilde Z}+\tilde B_{\tilde Y}$ and $Q=K_{\tilde Z}+\tau'_0 M_{\tilde Z}=K_{\tilde Z}+\tau'_0 p^*M_Z$. We have
\begin{equation}\label{eq: inequality}
\kappa_\sigma(K_{\tilde Y}+\tilde B_{\tilde Y}+\tau'_0 M_{\tilde Y}) \geq \kappa_\sigma((K_{\tilde Y}+\tilde B_{\tilde Y})|_{\tilde F})+ \kappa(K_{\tilde Z}+\tau'_0 M_{\tilde Z}),
\end{equation} where $\tilde F=\tilde Y_{z}$ is a general fiber.

\medskip

As $F = Y_z$ is also a general fiber of $g: Y \to Z$, and by \eqref{eq: resolution}, we have
\[
(K_{\tilde Y}+\tilde B_{\tilde Y})|_{\tilde F} = q_F^*((K_Y+B)|_{F})+E|_{\tilde F},
\] where $q_F: \tilde F \to F$ is the restriction of $q$ to $\tilde F$. By assumption, $(K_Y+B)|_{F}$ is pseudo-effective, and by above, $(K_{\tilde Y}+\tilde B_{\tilde Y})|_{\tilde F}$ is also pseudo-effective. By Lemma \ref{le: pseudo-effective}, we have $\kappa_{\sigma}(K_{\tilde Y}+\tilde B_{\tilde Y}+\tau'_0 M_{\tilde M}) = -\infty$ and $\kappa_{\sigma}((K_{\tilde Y}+\tilde B_{\tilde Y})|_{\tilde F}) \geq 0$.  Hence $\kappa(K_{\tilde Z}+\tau'_0 M_{\tilde Z}) = -\infty$ by \eqref{eq: inequality}. This is a contradiction because $M_{\tilde Z}$ is big and thus $\kappa(K_{\tilde Z}+\tau'_0 M_{\tilde Z}) = \kappa(K_{\tilde Z}+\tau_0 M_{\tilde Z}+(\tau_0'-\tau_0)M_{\tilde Z})= \dim \tilde Z > 0$.
\end{proof}

\begin{proof}[Proof of Theorem \ref{thm: main}] We divide the argument in several steps.

\medskip

\noindent Step 1. By adding $1$ to $I$, we can assume that $1 \in I$. Next, we claim that after replacing $I$ by $I_+$ (see \eqref{eq: I_+}), we can assume that $I=I_+$. In particular, the assumptions of Theorem \ref{thm: accumulation point of pet} is satisfied. Notice that this can only enlarge the upper bound. 

\begin{proof}[Proof of the Claim] If $I \subseteq [0,1]$ is DCC, then $I_+$ is also DCC. Moreover, $(I_+)_+ = I_+$. If $1$ is the only possible accumulation point of $I$, we claim that $1$ is
also the only possible accumulation point of $I_+$. Otherwise, there exists a sequence $\{c_+^{i}\}_{i\in \Nn}$ of $I_+$ approaching $c_+<1$. Each $c_+^i = \sum_{j}^{n_i} a_{ij}$, where $a_{ij} \in I$ (repetition is allowed). We claim that $n_i$ is bounded above. Otherwise there exists a subsequence $\{a_{k_i j_{k_i}}\}$ which decreases to $0$. By passing to a subsequence, we can assume that $n_i=n$ is a fixed number. For each $c_+^i$, we can associate it with an $n$-tuple $(a_{i1}, \ldots, a_{in})$ (the order does not matter). By passing to a subsequence again, we can assume that for each $k$, $\{a_{ik}\}_{i\in \Nn}$ is an increasing sequence. Hence for some $k$, there exists an accumulation point $\lim_i a_{ik} = a_k$, and $0<a_k<1$. This is a contradiction because $1$ is the only possible accumulation point in $I$.  
\end{proof}

If $k=n$, then $\lim^n (\PET_n(I)) \leq 1$ by \cite[Proposition 3.1]{HL18b}. Hence, we can assume $0<k<n$. By Theorem \ref{thm: accumulation point of pet}, we have $\lim^k (\PET_n(I)) \subseteq K_{n-k}(I)$.  Hence it is enough to give an upper bound for $K_{n-k}(I)$. Suppose that $0<\tau \in K_{n-k}(I)$, then by definition \eqref{eq: sets}, there exists $X'$ with $\rho(X')=1$, $\dim X' \leq n-k$, and $B', M'$ such that $({X'}, B'+\tau M')$ is g-klt satisfying condition $(\dagger)$. In this situation, we will show $\tau \leq \dim X'+1$. There are two cases to consider: either $M' \equiv 0$ or $M' \not\equiv 0$.

\medskip

\noindent Step 2. If $M' \equiv 0$, then some coefficient of $B'$ lies in $D_{\tau}(I)$. This coefficient is of the form 
\[
\frac{m-1+f+k\tau}{m}, \quad m,k\in \Nn \text{~and~} f\in I_+.
\] By generalized klt assumption, all the coefficients of $B'$ are less than $1$, hence $k\tau<1$, and thus $\tau<1$. The claim holds automatically.

\medskip

\noindent Step 3. If $M' \not\equiv 0$, then $M'$ is ample and $\Qq$-Cartier. By the definition of $\mathfrak{K}_n(I,c)$, there exists a log resolution $f: X \to X'$ and a semi-ample Cartier divisor $M$, such that $f_*(M) = M'$. Hence $f^*M' = M+E$ with $E \geq 0$ an $f$-exceptional divisor by the negativity lemma. Because $M' =f_*M \not\equiv 0$, $M$ is not numerically trivial. Hence some positive multiple of $M$ induces a fibration $\phi_M: X \to Z$ such that $M =\phi_M^*M_Z$ and $\dim Z>0$. By condition $(\dagger)$ (4), $M_Z$ is an ample \emph{Cartier} divisor. If $\phi_M$ is a birational morphism, that is, $M$ is nef and big, by Proposition \ref{prop: bound on Fano index}, we have $c\leq n-k+1$. Thus, we can assume that $0<\dim Z < \dim X$.

\medskip

We can write 
\begin{equation}\label{eq: on X}
K_X+\tilde B +\tau M =f^*(K_{X'}+B'+\tau M')+E,
\end{equation} where 
\[
\tilde B = f_*^{-1} B' + \sum_{a(F,X',B' + M')<1} (1-a(F,X',B' + M'))F
\] is a summation of the birational transform of $B'$ and $f$-exceptional divisors whose generalized log discrepancies are less than $1$. Thus $E \geq 0$ is an $f$-exceptional divisor such that $\tilde B, E$ have no common components. Because $({X'}, B'+\tau M')$ is g-klt, $(X, \tilde B +\tau M)$ is also g-klt with data $X \xrightarrow{\rm id} X \to \spec \Cc$, where $\tilde B$ is the boundary part and $\tau M$ is the nef part. In particular, $(X, \tilde B)$ is klt. By $K_{X'}+B'+\tau M' \equiv 0$, we have
\begin{equation}\label{eq: full equation 2}
K_X+\tilde B +\tau M \equiv E \geq 0.
\end{equation}

\medskip

We have $\tau = \pet(X, \tilde B; M)$. In fact, if $\tau>\tau_0 = \pet(X, \tilde B; M)$, then by $f_*\tilde B = B'$, $f_*(K_X+\tilde B +\tau_0 M) = K_{X'}+ B' + \tau_0 M'$ is pseudo-effective, which is a contradiction. Let $X_z$ be a general fiber of $X \to Z$, then
\[
(K_X+\tilde B)|_{X_z} \equiv (K_X+\tilde B +\tau M)|_{X_z} \equiv E|_{X_z}
\] is also pseudo-effective. By Proposition \ref{prop: num Iitaka}, we have $\tau \leq \pet(\tilde Z, 0; M_{\tilde Z})$, where $\tilde Z \to Z$ is a resolution of $Z$ and $M_{\tilde Z}$ is the pullback of $M_Z$. Because $M_{\tilde Z}$ is a nef and big Cartier divisor, by Proposition \ref{prop: bound on Fano index}, we have 
\[
\tau \leq \pet(\tilde Z,0; M_{\tilde Z}) \leq \dim \tilde Z+1 \leq \dim X' \leq n-k.
\] This completes the proof.
\end{proof}

\begin{proof}[Proof of Theorem \ref{thm: Fujita}]
By choosing $I = \{1\}$ in Theorem \ref{thm: main}, we get Theorem \ref{thm: Fujita}.
\end{proof}

We show by the following example that the upper bound in Theorem \ref{thm: main} is sharp for any $n$ and $0<k \leq n$.

\begin{example}\label{eg: the bound is sharp}
Let $I = \{\frac{l-1}{l} \mid l \in \Nn\}$. For $i\in \Nn$, consider the lc pair
\[
(\Pp^n, a^i_0 H_0+\cdots+a^i_{n}H_n),
\] where $a^i_j = \frac{l^i_j-1}{l^i_j} \in I$ and $H_j \in |\Oo_{\Pp^n}(1)|$ is a general divisor. Let $M \in |\Oo_{\Pp^n}(1)|$ be a Cartier divisor. Then 
\[
\pet(\Pp^n, a^i_0 H_0+\cdots+a^i_{n}H_n; M)=\sum_{j=0}^n \frac{1}{l^i_j}.
\] We have
\[
\{\sum_{j=0}^{n-k} \frac{1}{l^i_j} \mid l^i_j \in \Nn\} \subseteq \lim{}^k\PET(I),
\] and hence the upper bound $n-k+1$ is sharp.
\end{example}

\begin{remark}
Even though the bound in Theorem \ref{thm: main} is sharp (i.e. for coefficients in a DCC set), we do not know whether the bound in Theorem \ref{thm: Fujita} is sharp or not.
\end{remark}

\bibliographystyle{alpha}

\bibliography{bibfile}

\end{document}